\documentclass[11pt]{amsart}
\usepackage{amssymb}
\usepackage{amsmath}
\usepackage{setspace}
\usepackage{amsthm}
\usepackage{epsfig,color,colordvi,wasysym}
\usepackage{graphicx}
\usepackage{amsfonts}

\newtheorem{theorem}{Theorem}

\newtheorem{lemma}[theorem]{Lemma}
\newtheorem{corollary}[theorem]{Corollary}

\newtheorem{proposition}[theorem]{Proposition}

\newtheorem{conj}[theorem]{Conjecture}

\newtheorem{example}[theorem]{Example}

\newcommand\commentout[1]{}

\newcommand{\vol}{\operatorname{vol}}
\newcommand{\aff}{\operatorname{aff}}
\newcommand{\I}{\mathcal{I}}
\renewcommand{\P}{\mathcal{P}}

\newcommand{\NN}{\mathbb{N}}
\newcommand{\ZZ}{\mathbb{Z}}
\newcommand{\RR}{\mathbb{R}}

\renewcommand\emptyset{\varnothing}

\begin{document}

\title{The combinatorics of interval-vector polytopes}  

\author[Beck]{Matthias Beck}
\address{Department of Mathematics, San Francisco State University, San Francisco, CA 94132, USA}
\email{mattbeck@sfsu.edu}

\author[De Silva]{Jessica De Silva}
\address{Department of Mathematics, University of Nebraska, Lincoln, NE 68588, USA}
\email{jessica.desilva@huskers.unl.edu}

\author[Dorfsman--Hopkins]{Gabriel Dorfsman--Hopkins}
\address{Department of Mathematics, Dartmouth College, Hanover, NH 03755, USA}
\email{Gabriel.D.Dorfsman-Hopkins.13@dartmouth.edu}

\author[Pruitt]{Joseph Pruitt}
\address{Department of Mathematics, University of Illinois, Urbana--Champaign, IL 61801, USA}
\email{j92pruitt@gmail.com}

\author[Ruiz]{Amanda Ruiz}
\address{Department of Mathematics, Harvey Mudd College, Claremont, CA 91711, USA}
\email{amruiz@hmc.edu}



\date{11 August 2013}

\thanks{We thank the creators and maintainers of the software packages {\tt
polymake} \cite{polymake} and {\tt LattE} \cite{latte,koeppelatte}, which were
indispensable for our project.
We also thank an anonymous referee for helpful suggestions, and Ricardo Cortez and the staff at MSRI for creating an ideal research environment at MSRI-UP. 
This research was partially supported by the NSF through the grants DMS-1162638 (Beck) and DMS-1156499 (MSRI-UP REU), and by the NSA through grant H98230-11-1-0213.}

\maketitle

\begin{abstract}
An \emph{interval vector} is a $(0,1)$-vector in $\mathbb{R}^n$ for which all the $1$'s appear consecutively, and an \emph{interval-vector polytope} is the convex hull of a set of
interval vectors in $\mathbb{R}^n$.
We study three particular classes of interval vector polytopes which exhibit interesting
geometric-combinatorial structures; e.g., one class has volumes equal to the Catalan
numbers, whereas another class has face numbers given by the Pascal 3-triangle.
\end{abstract}


\section{Introduction}

An \emph{interval vector} is a $(0,1)$-vector $x\in\mathbb{R}^n$ such
that, if $x_i = x_k = 1$ for $i<k$, then $x_j = 1$ for every $i\le j\le k$. 
In \cite{MR2825295} Dahl introduced the class of \emph{interval-vector polytopes}, 
 which are formed by taking the convex hull of a set of interval vectors in
$\mathbb{R}^{n}$. Our goal is to derive combinatorial properties of certain interval-vector polytopes.

For $i\le j$, let
$\alpha_{i,j}:=e_i + e_{i+1} + \dots + e_j$, where $e_i$ is the
$i^{\textit{th}}$ standard unit vector.  The \emph{interval length} of
$\alpha_{ij}$ is $j-i+1$. 
 Let $S\subset \NN$. For a fixed $n$, let $\I_S$ be the set of interval vectors in $\RR^n$
with interval length in $S$. (If $S$ is small, we may leave out the brackets in
the set notation; e.g., we will denote $\I_{\{i,j\}}$ by $\I_{i,j}$.) We will
denote the set of all non-zero interval vectors in a given dimension as $\I_{[n]}$.
Let $\P_n(\I_S)$ be the convex hull of $\I_S\subset \RR^n$.
 
There are three classes of interval vector polytopes that we will consider in this
paper. In Section~\ref{complete} we study the \emph{complete interval vector polytope}
$\P_n(\I_{[n]})$, the convex hull of all interval vectors in $\RR^n$ except the zero vector. In Section \ref{fixed} we look at the \emph{fixed interval vector
polytope} $\P_n(\I_i)$ given by the convex hull of all interval vectors with
interval length $i$. In Section \ref{pyramid} we  introduce the first in a class
of \emph{pyramidal interval polytopes}: the \emph{first pyramidal interval vector polytope} $\P_n(\I_{1,n-1})$, the convex hull of all interval vectors in $\RR^n$ with interval length $1$ or $n-1$. 
(The reason for the term \emph{pyramidal interval polytope} will also become clear in Section \ref{pyramid}.)
In Section \ref{conjectures} we generalize this to the \emph{$i^\text{th}$ pyramidal
interval vector polytope} $\P_n(\I_{1,n-i})$.  We examine combinatorial characteristics of these polytopes such as the $f$-vector and volume and discover  unexpected relations to well-known numerical sequences. 

Let $t$ be a positive integer variable. For a lattice polytope $\P$ (i.e., the
vertices of $\P$ all have integer coordinates), the \emph{Ehrhart polynomial}
$L_{\mathcal{P}}(t)$ is the counting function 
yielding the number of lattice points in 
$t\mathcal{P}:=\{tv\;|\;v\in\mathcal{P}\}$.
Ehrhart \cite{ehrhartpolynomial} proved that $L_{\mathcal{P}}(t)$ is indeed a
polynomial; see, e.g., \cite{ccd} for more about Ehrhart polynomials.
The  Ehrhart polynomial  contains useful geometric information about a polytope;
in particular, the leading coefficient of the Ehrhart polynomial gives the volume of the polytope.

In \cite{MR2487491}, Postnikov defines the \emph{complete root polytope}
$Q_n\subset\mathbb{R}^n$ as the convex hull of $0$ and $e_i-e_j$ for all $i<j$ where $e_i$ is the $i^{\text{th}}$ standard unit vector. 
He showed (among many other things) that the volume of $Q_{n+1}$ is $C_n:= \frac{1}{n+1}{2n\choose n}$, the $n^{\textrm{th}}$ Catalan number.
In Section \ref{complete} we prove, in a discrete-geometric
sense, that $Q_{n+1}$ and the complete interval vector polytope $\P_n(\I_{[n]})$ are interchangeable, that is, the two polytopes
have the same Ehrhart polynomial.

\begin{theorem}\label{CompleteRootEquiv}
$L_{Q_{n+1}}(t) = L_{\P_n(\I_{[n]})}(t) \, .$
\end{theorem}

\begin{corollary}\label{CIPvolume}
The volume of the complete interval vector polytope $\P_n(\I_{[n]})$ equals 
the $n^{\textrm{th}}$ Catalan number.
\end{corollary}

A \emph{unimodular simplex} in $\RR^d$ is an $n$-dimensional lattice simplex $\Delta$ whose edge direction at any vertex form a lattice basis for $\ZZ^d \cap
\aff(\Delta)$, where $\aff(\Delta)$ is the affine hull of~$\Delta$. 
In Section \ref{fixed} we prove: 
 
\begin{theorem}\label{unimodular}
The fixed interval vector polytope $\P_n(\I_{i})$ is an $(n-i)$-dimensional unimodular simplex.
\end{theorem}
 
Given an $n$-dimensional polytope
$\mathcal{P}$ with $f_{k}$ $k$-dimensional faces, the \emph{$f$-vector} of
$\mathcal{P}$ is written as
$f(\mathcal{P}):=(f_{-1},f_{0},f_{1},\dots,f_{n})$ where $f_{-1},f_n:=1$ (see,
e.g., \cite{MR1976856} for more about $f$-vectors).
In Section~\ref{pyramid} we show:

\begin{theorem}\label{fvector}
For $n\ge 3$, the $f$-vector of the first pyramidal interval vector polytope
satisfies $f_k(\P_n(\I_{1,n-1})) = \binom{ n-1 }{ k } + \binom{ n+1 }{ k+1 } \,
.$
\end{theorem}
The $f$-vector of $\P_n(\I_{1,n-1})$ is thus the $n^\text{th}$ row of the
\emph{Pascal 3-triangle} (see, e.g., \cite[Sequence A028262]{sloaneonlineseq}), in particular, it is symmetric. 
We also show that the volume of the $1^\text{st}$ pyramidal interval vector polytope is simple:

\begin{theorem}\label{volpyramid}
For $n \geq 3$, $ \vol ( \P_n(\I_{1,n-1})) = 2(n-2) \, .$
\end{theorem}

Finally, in Section \ref{conjectures} we lay out future work on $i^\text{th}$
pyramidal interval vector polytopes.


\section{Preliminaries}

In this paper, we will be analyzing the properties of certain classes of
\emph{convex polytopes} which are formed by taking the convex hull of finitely
many points in $\mathbb{R}^{n}$. The \emph{convex hull} of a set 
$A=\{v_{1},v_{2},\dots,v_{m}\}\subset\mathbb{R}^{n}$, denoted
$\operatorname{conv}(A)$, is defined as

\begin{equation}\label{hull} \left\{\lambda_{1}v_{1}+\lambda_{2}v_{2}+\dots+\lambda_{m}v_{m}
\mid\lambda_{1},\lambda_{2},\dots,\lambda_{m}\in\mathbb{R}_{\ge
0}\ \textrm{ and }\ \sum_{i=1}^{m}\lambda_{i}=1\right\}.\end{equation} 
The
polytope $\operatorname{conv}(A)$ is contained in  the \emph{affine hull} $\operatorname{aff}(A)$ of
$A$, defined as in \eqref{hull} but without the restriction
that $\lambda_{1},\lambda_{2},\dots,\lambda_{m}\ge 0$. We call a set of points
\emph{affinely} (resp.\ \emph{convexly}) \emph{independent} if each point is not
in the affine (resp.\ convex) hull of the rest. The \emph{vertex set} of a polytope is the minimal convexly
independent set of points whose convex hull form the polytope. 
A polytope is \emph{d-dimensional} if the dimension of its affine hull
is $d$. We denote the dimension of the polytope $\mathcal{P}$ as
$\operatorname{dim}(\mathcal{P})$. We call a $d$-dimensional polytope a
\emph{d-simplex} if it has $d+1$ vertices.

A \emph{lattice point} is a point with integral coordinates. A \emph{lattice
polytope} is a polytope whose vertices are lattice points.
The \emph{normalized volume} of a polytope $\mathcal{P}$, denoted $\operatorname{vol}(\mathcal{P})$, is the volume with respect to a unimodular simplex (recall definition in Section 1). We will refer to the normalized volume of a polytope as its \emph{volume}.
Note that the leading coefficient of the Ehrhart polynomial of a lattice polytope
$\P$ is $\frac{1}{d!} \vol(\P)$. 

A \emph{hyperplane} is a set of the form
\[H:=\left\{x\in\mathbb{R}^{n} \, | \, a_{1}x_{1}+\cdots+a_{n}x_{n}=b\right\},\] where not all
$a_{j}$'s are 0. The \emph{half-spaces} defined by this hyperplane are formed by
the two weak inequalities corresponding to the equation defining the hyperplane.
A \emph{face} of
$\mathcal{P}$ is the intersection of a hyperplane and $\mathcal{P}$ such that
$\mathcal{P}$ lies completely in one half-space of the hyperplane. This face is
a polytope called a \emph{k-face} if its dimension is $k$. A vertex is a
$0$-face and an \emph{edge} is a $1$-face. Given a $d$-dimensional polytope
$\mathcal{P}$ with $f_{k}$ $k$-dimensional faces, the \emph{$f$-vector} of
$\mathcal{P}$ is written as
$f(\mathcal{P}):=(f_{-1},f_{0},\dots,f_{n})$. For example, a triangle
$\triangle$  is a $2$-dimensional polytope with $3$ vertices and $3$ edges and thus has $f$-vector $f(\triangle)=(1,3,3,1)$. 


\section{Complete Interval Vector Polytopes}\label{complete}

In \cite{MR2825295} Dahl provides a method for determining
the dimension of these polytopes which we will use throughout this paper. 
We utilized the software packages {\tt polymake} \cite{polymake} and {\tt LattE}
\cite{latte,koeppelatte} to find most of the patterns described by our results.

\begin{proof}[Proof of Theorem \ref{CompleteRootEquiv}] Each of the vertices of $Q_n$ are vectors with entries that sum to
zero, so any linear combination (and specifically any convex combination) of
these vertices also has entries who sum to zero.
Define $B := \{x\in\mathbb{R}^n\, | \, \sum_{i=1}^n x_i = 0\}$; thus $Q_n\subset B$, and 
$B$ is an $(n-1)$-dimensional affine subspace of~$\mathbb{R}^n$.

Consider the linear transformation $T$ given by the $n\times n$ lower triangular
matrix with entries $t_{i,j} = 1$ if $i\ge j$ and $t_{i,j} = 0$ otherwise.  Then
 \[T(B) \subseteq A := \left\{x\in\mathbb{R}^n \, | \, x_n=0\right\}.\]  Since (the matrix representing) $T$ has
determinant 1, it is injective when restricting the domain to $B$.  For the same
reason, we know that for any $y\in A$, there exists $x\in\mathbb{R}^n$ such that
$y=T(x)$.  But since $y_n = \sum_{i=1}^n x_i = 0$, then $x\in B$, so that
$T|_B:B\to A$ is surjective, and therefore a linear bijection.

 Also, the projection $\Pi:A\to \mathbb{R}^{n-1}$ given by
\[\Pi\left((x_1,\ldots, x_{n-1}, 0)\right) = (x_1,\ldots,x_{n-1}),\]
 is clearly a linear bijection.

Now we show that the linear bijection $\Pi\circ T|_{B}: B\to\mathbb{R}^{n-1}$ is
a lattice-preserving map, i.e., an isomorphism from $B \cap \mathbb{Z}^n$ to $\mathbb{Z}^{ n-1 }$ (viewed
as additive groups).  First we find a lattice basis for $B$.  Consider \[C
:= \left\{e_{i,n} = e_i-e_n \, | \, i<n \right\}.\]  We notice that any integer point of $B$ can be
represented as \[\left(a_1,\ldots, a_{n-1}, -\sum_{i=1}^{n-1}a_i\right) =
\sum_{i=1}^{n-1}a_ie_{i,n}\]  and so $C$ is a lattice basis.

Note that $\Pi\circ T(e_{i,n}) = e_i + \cdots + e_{n-1} =: u_i$.  Therefore
\[\Pi\circ T(C) = \{u_i \, | \, i\le n-1\} =: U \, .\]  We notice that $e_{n-1} = u_{n-1}$
and $e_i = u_i - u_{i+1}$, so that each of the standard unit vectors $e_i$ of
$\mathbb{R}^{n-1}$ is an integral combination of the vectors in $U$.  Since the
standard basis is a lattice basis, so is $U$, thus $\Pi\circ T|_{B}$ is a
lattice-preserving map.
Since our bijection is linear and lattice-preserving, all we have left to show
is that the vertices of $Q_n$ map to those of $\P_{n-1}(\I_{[n-1]})$.  By
linearity, $\Pi\circ T(0) = 0$, and  given any vertex $\alpha_{i,j}$ of
$\P_{n-1}(\I_{[n-1]})$, we know that $\Pi\circ T(e_{i,j+1}) = \alpha_{i,j}$
where $i< j+1 \le n$ so that $\Pi\circ T|_{B}$ maps vertices to vertices.
\end{proof}
Corollary \ref{CIPvolume} follows directly from this theorem and \cite{MR2487491}, since the leading coefficient of the Ehrhart polynomial of $\P_{n}$ is $\frac{1}{n!}$ times the volume of~$\P_{n}$.


\section{Fixed Interval Vector Polytopes}\label{fixed}

The following construction is due to \cite{MR2825295}. 
We define the set of \textit{elementary vectors} as containing
all $e_{i,j} = e_i - e_j$, each unit vector $e_i$, and the zero vector.  Let $T$
be the lower triangular matrix from the proof of Theorem
\ref{CompleteRootEquiv}.  We notice that $T(e_i) = \alpha_{i,n}$ and $T(e_{i,j})
= \alpha_{i,j-1}$.  So the image of an elementary vector is an interval vector. 
Since $T$ is invertible, for any set of interval vectors $\mathcal{I}$, there is
a unique set $\mathcal{E}$ of elementary vectors such that $T(\mathcal{E}) =
\mathcal{I}$, namely $  \mathcal{E}=T^{-1}(\mathcal{I})$.

Thus for any interval vector polytope $\mathcal{P}_n(\mathcal{I}_S)\subset\mathbb{R}^{n}$,
we can construct the corresponding \emph{flow-dimension graph} $G_{\mathcal{I}_S}
= (V,E)$ as follows.  Let $\mathcal{E}_S = T^{-1}(\mathcal{I}_S)$.  Let the
vertex set $V = [n]$. Specify a subset $V_1 = \{j\in V \;  | \; 
e_j\in\mathcal{E}_S\}$, and define the directed edge set $E = \{(i,j) \;  | \; 
e_{i,j}\in\mathcal{E}_S\}$. Let $k_0$ denote the number of connected
components $\mathcal{C}$ of the graph $G$ (ignoring direction) so that
$\mathcal{C}\cap V_1 $ is empty.

Recall that the fixed interval vector polytope $\P_n(\I_{i})$
is the convex hull of all interval vectors in $\mathbb{R}^{n}$ with interval length $i$.
For example, the fixed interval vector polytope with $n=5$, $i=3$ is
$$\P_5(\I_{3}) = \operatorname{conv} \big( (1,1,1,0,0) \, , \, (0,1,1,1,0)
\, , \, (0,0,1,1,1) \big)$$
and its flow-dimension graph is depicted in Figure~\ref{flowdim}.

\begin{figure}[htb]
\begin{center}
\includegraphics[scale=.5]{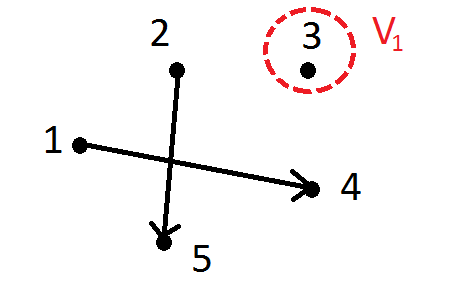}
\end{center}
\caption{The flow-dimension graph of $\P_5(\I_{3})$.}\label{flowdim}
\end{figure}

\begin{theorem}[Dahl
 \cite{MR2825295}]\label{graphdimension} 
If $0\in\operatorname{aff}(\I_S)$, then the dimension of $\P_n(\I_S)$ is
$n-k_0$.  Else, if $0\notin\operatorname{aff}(\mathcal{I}_S)$ then the dimension
of $\P_n(\I_S)$ is $n-k_0-1$.
\end{theorem}
For a fixed $i$,
\[
 T^{-1}(\I_{i})=\mathcal{E}_i = \{e_{k,k+i}\;|\;k\le n-i\}\cup\{e_{n-i+1}\} \, .
\]
The corresponding
flow-dimension graph is $G_{\P_n(\I_{i})}=(V,E)$ where $V=\{1,\ldots,n\}$ and
$E = \{ (k,k+i)\,|\,k\in[n-i] \}$.
Then $V_{1}=\{n-i+1\}$ corresponds to $e_{n-i+1}\in\mathcal{E}_i$.

Two nodes $a,b$ in a graph $G=(V,E)$ are said to be \emph{connected} if there
exists a \emph{path} from $a$ to $b$, that is there exist $q_{0}, \dots,
q_{s}\in V$ such that $(a,q_{0}),(q_{0},q_{1}),\dots, (q_{s},b) \in E$.

\begin{lemma}\label{path}
Let $a,b$ be nodes in the flow-dimension graph $G_{\P_n(\I_{i})}$.
Then $a$ and $b$ are connected if and only if $a \equiv b \; \operatorname{mod}i$.
\end{lemma}

\begin{proof}
The edges in $G_{\P_n(\I_{i})}$ are of the form $(k, k+i)$, and therefore the nodes of a path in
$G_{\P_n(\I_{i})}$ are all in the same congruence class modulo~$i$.
\commentout{
Assume without loss of generality $a< b$. Suppose $a$ and $b$ are connected by
the path $q_{0}, \dots, q_{s}\in V$. Therefore by definition of $E$, we have
\begin{align*}
q_{0} &= a + i \\
q_{1} &= q_{0} + i = a + 2i \\
&\ \vdots \\
q_{s} &= q_{s-1} + i = a + (s+1)i \\
b &= q_{s} + i = a + (s+2)i \, .
\end{align*}
Thus $a \equiv b \; \operatorname{mod}i$ by definition. 

Conversely, suppose that $a \equiv b \; \operatorname{mod}i$. Then there
exists $m \in \mathbb{N}$ such that
\[
b = a + mi 
  = a + (m-1)i + i \, .
\]
Since $b$ and $a+(m-1)i$ differ by $i$, then by definition of $E$, there is an
edge between these nodes. Call this edge $(q_{t},b) \in E$. Similarly, we have
\begin{align*}
a + (m-1)i &= a+(m-2)i + i &\Rightarrow& \qquad (q_{t},q_{t-1}) \in E \\
a + (m-2)i &= a+(m-3)i + i &\Rightarrow& \qquad (q_{t-1},q_{t-2}) \in E \\
 \vdots \\
a + 2i &= (a+i) + i  &\Rightarrow& \qquad (q_{1},q_{0}) \in E \\
a + i &= a + i &\Rightarrow& \qquad (q_{0},a) \in E \, .
\end{align*}
Hence $q_0,q_1,\ldots,q_t$ define a path from $a$ to $b$, so $a$ and $b$
are connected.
}
\end{proof}

\begin{proposition}
$\P_n(\I_{i})$ is an $(n-i)$-dimensional simplex.
\end{proposition}

\begin{proof}
For a given dimension and interval length, an interval vector is uniquely
determined by the location of the first 1, hence we can determine the number of
vertices of $\P_n(\I_{i})$ by counting all possible placements of the first
1 in an interval of $i$ 1's.  Since the string must have length $i$, the number
of spaces before the first 1 must not exceed $n-i$ and so there are $n-i+1$
possible locations for the first 1 in the interval to be placed. Thus,
$\P_n(\I_{i})$ has $n-i+1$ vertices.

By Lemma \ref{path} we know there are $i$ connected components in the
flow-dimension graph $G_{\P_n(\I_{i})}$ and since $V_{1}$ has only one
element, $k_{0} = i - 1$. Thus by Theorem \ref{graphdimension} the dimension of
$\P_n(\I_{i})$ is $n - i$. Therefore $\P_n(\I_{i})$ is an
$(n-i)$-dimensional simplex.
\end{proof}

\begin{proof}[Proof of Theorem \ref{unimodular}]
It remains to show that $\P_n(\I_{i})$ is unimodular.
Consider the affine space where the sum over every $i^{\textrm{th}}$ coordinate
is 1,
\[A=\left\{\mathbf{x}\in\mathbb{R}^{n}\;\Bigg|\;\displaystyle{\sum_{j\equiv
k\operatorname{mod}i}}x_{j}=1,\;\text{ for all }\,k\in[i]\right\}.\]
Since the vertices of $\P_n(\I_{i})$ have interval length $i$, they are in
$A$; thus $\P_n(\I_{i})\subset A$. 
We want to show that the following vectors in $\P_n(\I_{i})$ form a
lattice basis for $A$: 
\[
\begin{array}{lrl}
w_1 &=& \alpha_{1,i}- \alpha_{n-i+1,n}\\
w_2 &=& \alpha_{2,i+1}-\alpha_{n-i+1,n}\\
&\vdots&\\
w_{n-i} &=& \alpha_{n-i,n-1}-\alpha_{n - i+1,n} \, . \end{array}\]
We will do this by showing that any integer point $p\in A$ can be expressed as an
integral linear combination of the proposed lattice basis, that is, there exist
integer coefficients $Y_1,\ldots,Y_{n-i}$ so that $p=Y_1w_1 + \cdots +
Y_{n-i}w_{n-i} + \alpha_{n-i+1,n}$.

We first notice that $p$ can be expressed as
\[\left(p_1, p_2, \ldots, p_{n-i}, \sum_{\underset{j\equiv
t-i+1\operatorname{mod}i}{j\le n-i}}(-p_j) + 1,\sum_{\underset{j\equiv
t-i+2\operatorname{mod}i}{j\le n-i}}(-p_j)+ 1, \ldots,\sum_{\underset{j\equiv t
\operatorname{mod}i}{j\le n-i}}(-p_j) + 1\right)\]
by solving for the last term in each of the equations defining $A$. Let
\[Y_t = \left\{\begin{array}{lrl}
	p_1&\text{ if }&t = 1,\\
	p_t - p_{t-1}&\text{ if }&1<t\le i,\\
	p_t - Y_{t-i}&\text{ if }&i<t\le n-i.
\end{array}\right.\]
 Then each $Y_t$ is an integer.  We claim that \[Y_1w_1 + \dots + Y_{n-i}w_{n -
i} + \alpha_{n-i+1,n} = p.\] Clearly the first coordinate is $p_1$ since $w_1$
is the only vector with an element in the first coordinate.  Next consider the
$t^{\text{th}}$ coordinate of this linear combination for $1<t\le i$, by summing
the coefficients of all the vectors who have a 1 in the $t^{\text{th}}$
position:
\[Y_t + Y_{t-1} + Y_{t-2} + \dots + Y_1 = p_t - p_{t-1} + p_{t-1} - p_{t-2} +
\cdots + p_{2} - p_{1} + p_1 = p_t\]
We next consider the $t^{\text{th}}$ coordinate of the combination for $i<t\le
n-i$ by summing the coefficients of the vectors who have a 1 in the
$t^{\text{th}}$ position.
\[Y_t + Y_{t-1} + \cdots + Y_{t-i+1} = (p_t - Y_{t-1} - \dots - Y_{t-i+1}) +
Y_{t-1}  + \dots + Y_{t-i+1} = p_t\]
Finally, we consider the $t^{\text{th}}$ coordinate of the combination for
$n-i<t\le n$, noticing that each coordinate from $w_1$ to $w_t$ has a $-1$ in
the $(t-i)^{\text{th}}$ position, and $\alpha_{n-i+1,n}$ has a 1 in this
position.  This gives 
\[-(Y_1 + Y_2 + \dots + Y_{t-i}) + 1.\]
Applying the two relations we have defined between coordinates, and calling
$\langle t\rangle$ the least residue of $t\operatorname{mod}i$, we see that
\begin{eqnarray*}
-(Y_1 + Y_2 + \dots + Y_{t-i}) + 1 &=& -(Y_1 + Y_2 + \dots + Y_{t-2i} + p_{t-i})
+ 1\\
&=& - (Y_1 + Y_2 + \dots + Y_{t-3i} + p_{t-2i} + p_{t-i}) + 1\\
&=& - \left(Y_1 + Y_2 + \dots + Y_{\langle t\rangle} + \sum_{\underset{j\equiv
t\operatorname{mod}i}{i<j\le n-i}}p_j\right) + 1\\
&=&-\left(\sum_{\underset{j\equiv t\operatorname{mod}i}{j\le n-i}}p_j\right) +
1.
\end{eqnarray*}
Thus $p = Y_1w_1 + Y_2w_2 + \dots + Y_{n-i}w_{n-i} + \alpha_{n-i+1,n}$ and so
$w_1,\ldots,w_{n-i}$ form a lattice basis of $A$.
Thus the vertices of $\P_n(\I_{i})$ form a lattice basis, and so $\P_n(\I_{i})$ is a unimodular simplex.
\end{proof}

\section{The first pyramidal interval vector polytope}\label{pyramid}

Recall that $\P_{n}(\I_{1,n-1})$ is the convex hull of all
vectors in $\mathbb{R}^{n}$ with interval length 1 or $n-1$.
For example,
$$ \P_{4}(\I_{1,3}) = \operatorname{conv} \big( (1,0,0,0) \, , \,(0,1,0,0) \, ,
\,(0,0,1,0) \, , \,(0,0,0,1) \, , \,(1,1,1,0) \, , \,(0,1,1,1) \big) \, , $$
whose flow-dimension graph is depicted in Figure~\ref{firstpyramidflowgraph}.

\begin{figure}[htp]
\begin{center}$\includegraphics[scale=.3]{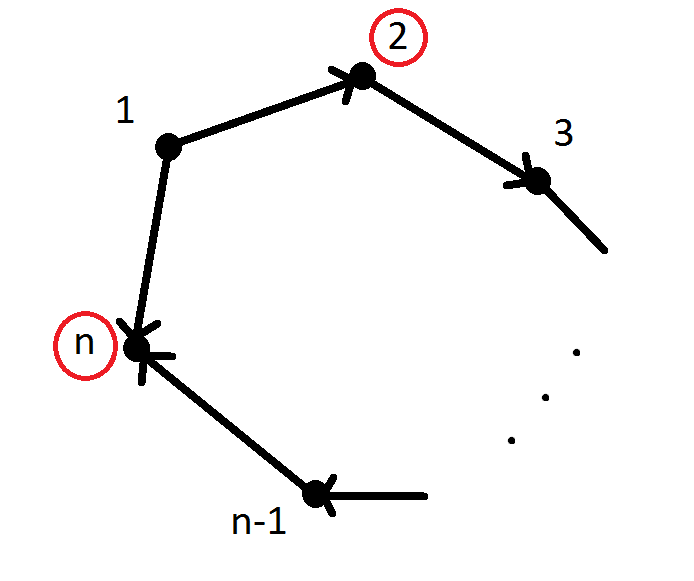}$\end{center}
\caption{$G_{\P_{n}(\I_{1,n-1})}$.}\label{firstpyramidflowgraph}
\end{figure}

\begin{proposition}\label{dimension}
The dimension of $\P_{n}(\I_{1,n-1})$ is $n$.
\end{proposition}

\begin{proof}
The affine hull of $e_1, \dots, e_n$ is the $(n-1)$-dimensional set
\[
  \left\{ x \in \RR^n \, | \, x_1 + \dots + x_n = 1 \right\} . 
\]
Since $\alpha_{ 1,n-1 }$ is not in this set, $\dim(\P_{n}(\I_{1,n-1})) = n$. 
\commentout{
For $n\ge 3$, the vertices of $\P_{n}(\I_{1,n-1})$ form the set
$$\mathcal{I}=\left\{\begin{matrix}
e_1 &= (1, 0,\dots, 0, 0)\\
e_2 &= (0, 1,\dots, 0, 0)\\
\vdots\\
e_n &= (0, 0, \dots, 0, 1)\\
\alpha_{1,n-1} &= (1, 1,\dots, 1, 0)\\
\alpha_{2,n} &= (0, 1,\dots, 1, 1)
\end{matrix}\right\}.$$ We convert the interval vectors to the corresponding
elementary vector set \[\mathcal{E} = \{e_{1,2}, e_{2,3}, \ldots, e_{n-1,n},
e_{1,n}, e_2, e_n\}.\]  From this we construct the flow-dimension graph    
$G_{\P_{n}(\I_{1,n-1})} = (V,E)$ as seen in Figure 1, where $V = [n]$ and \[E =
\{(k,k+1)|k\in[n-1]\}\cup\{(1,n)\}\] corresponding to each $e_{i,j}$ in
$\mathcal{E}$.  The subset of vertices $V_1 = \{2, n\}$ (circled in Figure 1)
corresponds to each $e_i$ in $\mathcal{E}$.  Since the underlying graph is
connected, we know \[k_0=\;\#\{\textrm{connected components }C\text{ in
}G_{\P_{n}(\I_{1,n-1})}\text{ such that }C\cap V_1=\emptyset\} = 0.\]
Next we notice that
\[\frac{1}{n-2}e_1 + \frac{1}{n-2}e_2+\dots+\frac{1}{n-2}e_{n-1} -
\frac{1}{n-2}\alpha_{1,n-1} = \mathbf{0}\]
where the sum of the coefficients is
\[\frac{n-1}{n-2} - \frac{1}{n-1} = \frac{n-2}{n-2} = 1\]
So $\mathbf{0}\in\operatorname{aff}(\mathcal{I})$ and by Theorem
\ref{graphdimension}, $\operatorname{dim}(P_{n,1}) = n-k_0 = n$.
}
\end{proof}


Recall that the $f$-vector of a polytope tells us the number of faces the polytope
has of each dimension. Our next task is to compute the $f$-vector of
$\P_{n}(\I_{1,n-1})$.

\commentout{
with $n\ge 3$, is precisely the $n^{\textrm{th}}$ row of the Pascal 3-triangle
without 1's. The \emph{Pascal 3-triangle} is an analogue of Pascal's Triangle,
where the third row, instead of being 1 2 1, is replaced with 1 3 1, and then
the same addition pattern is followed as in Pascal's triangle.
}

\begin{lemma}\label{subdimension}
Let  $n\ge 3$.  Then
$\mathcal{B} := \operatorname{conv}(e_1, e_n, \alpha_{1,n-1},\alpha_{2,n})$ is a
2-dimensional face of $\P_{n}(\I_{1,n-1})$.
\end{lemma}

\begin{figure}[htb]
\label{Figure 2}
\begin{center}
\includegraphics[scale=.3]{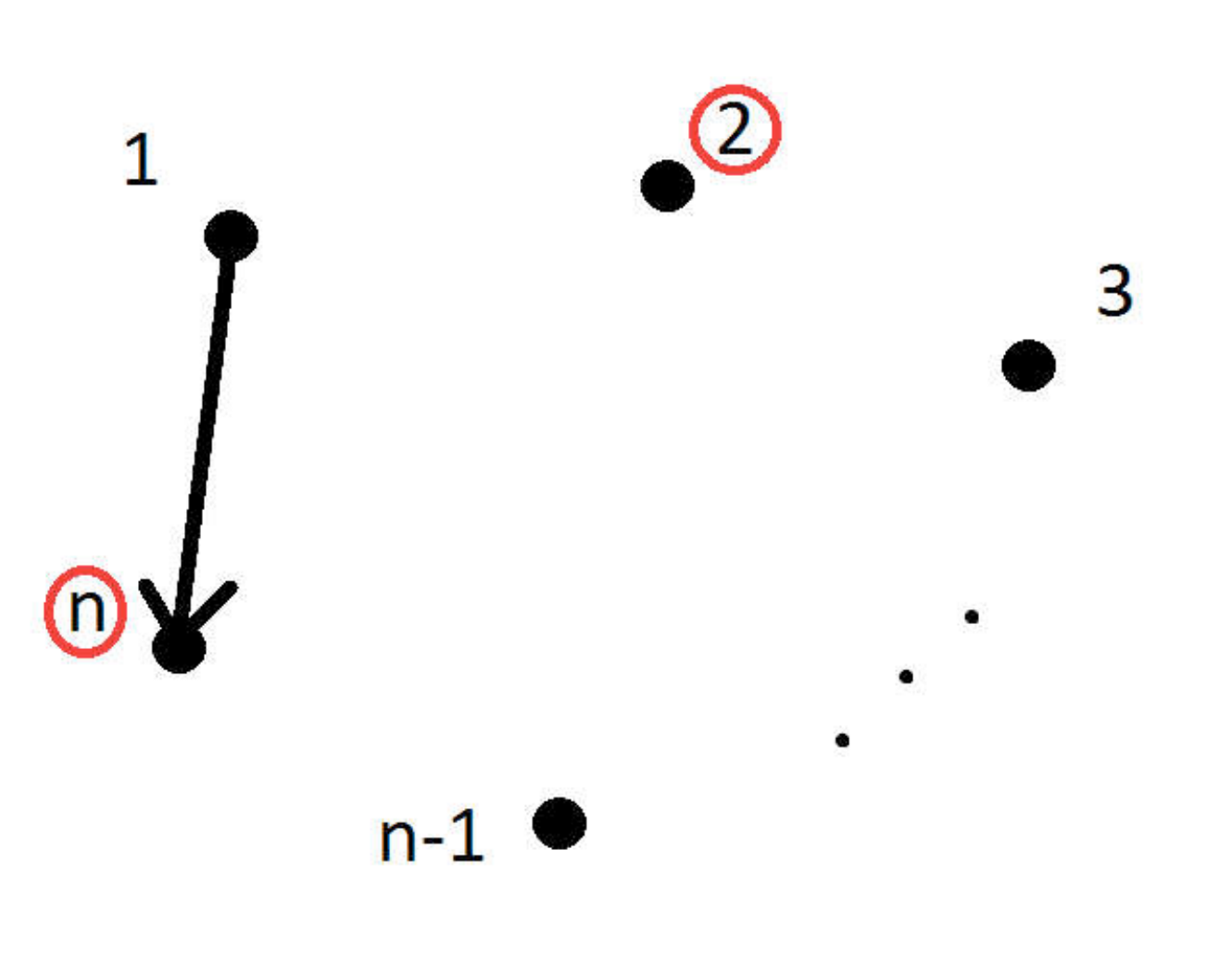}
\end{center}
\caption{$G_{\mathcal{A}}$.}
\end{figure}

\begin{proof}
We first consider
$\mathcal{A}=\operatorname{conv}(e_n,\alpha_{1,n-1},\alpha_{2,n})$.  The
corresponding elementary vectors of the vertex set are $\{e_{1,n}, e_2, e_n\}$. 
So we build the flow-dimension graph as seen in Figure 2, $G_{\mathcal{A}} =
(V,E)$ where $V = [n]$, $E = \{(1,n)\}$ corresponding to $e_{1,n}$. The subset
$V_1 = \{2, n\}$ (circled in Figure 2) corresponds to $e_2$ and $e_n$.  This
graph has $n-1$ connected components, two of which contain elements of $V_1$ so
that $k_0 = n-3$.

If we let $\lambda_1 e_n + \lambda_2 \alpha_{1,n-1} + \lambda_3 \alpha_{2,n} =
\mathbf{0}$, we first notice that $\lambda_2 = 0$ since $\alpha_{1,n-1}$ is the
only vector with a nonzero first coordinate.  But this implies that $\lambda_1 =
\lambda_3 = 0$.  Since the coefficients cannot sum to 1, we conclude that
$\mathbf{0}\notin\operatorname{aff}(e_n, \alpha_{1,n-1}, \alpha_{2,n})$.
So now by Theorem \ref{graphdimension},
\[\operatorname{dim}(\operatorname{conv}(e_n, \alpha_{1,n-1}, \alpha_{2,n})) = n
- k_0 - 1 = n-(n-3) - 1 = 2.\] Finally $e_1 = (1) \alpha_{1,n-1} + (-1)
\alpha_{2,n} + (1)e_n$ is in the affine hull of $\mathcal{A}$ and thus does not add a dimension.  We conclude that $\operatorname{dim}(\mathcal{B}) = 2$.
\end{proof}

\begin{corollary}\label{dimsum}
Let  $\mathcal{I}:=
\{e_1,e_2,\ldots,e_n,\alpha_{1,n-1},\alpha_{2,n}\}$ .  For $2\le
i\le n-1$  each $e_i$ adds a dimension to $\P_{n}(\I_{1,n-1})$, that is,
$e_i\notin\operatorname{aff}(\mathcal{I}\setminus\{e_i\})$.
\end{corollary}
\begin{proof}
This follows from Proposition \ref{dimension} and Lemma \ref{subdimension}.  Since
$\mathcal{B}$ has dimension $2$ and $\P_{n}(\I_{1,n-1})$ has dimension $n$, then
the $n-2$ remaining vertices must add the remaining $n-2$ dimensions. 
\end{proof}

\begin{lemma}\label{basefvector} Let $\mathcal{B}$ as in Lemma
\ref{subdimension}. Then $\mathcal{B}$ has $f$-vector $(1,4,4,1)$.
\end{lemma}

\begin{proof}
Since $\mathcal{B}$ has dimension 2, $f_1 = f_0$.  We know that
$\{e_n,\alpha_{1,n-1},\alpha_{2,n}\}$ are three vertices of $\mathcal{B}$.  If
$e_1\in\operatorname{conv}(e_n,\alpha_{1,n-1},\alpha_{2,n})$ then
\begin{equation}\label{e}e_1 = \lambda_1 e_n + \lambda_2 \alpha_{1,n-1} +
\lambda_3 \alpha_{2,n}\end{equation} where the coefficients sum to 1.  Since
$\alpha_{1,n-1}$ is the only vector with a nonzero coordinate in the first
position,  $\lambda_2 = 1.$ This in turn implies that $\lambda_1 =
\lambda_3 = 0$, contradicting (\ref{e}).  So
$e_1\notin\operatorname{conv}(e_n,\alpha_{1,n-1},\alpha_{2,n})$ and therefore
forms a fourth vertex. 
\end{proof}
We can tie all this together with the following theorem.  First we define a
\textit{d-pyramid} $P$ as the convex hull of a
$(d-1)$-dimensional polytope $K$ (the \textit{basis} of $P$) and a point
$A\notin\operatorname{aff}(K))$ (the \textit{apex} of $P$).
\begin{theorem}[see, e.g., \cite{MR1976856}]\label{facenumbers}
If $P$ is a $d$-pyramid with  basis $K$ then
\begin{eqnarray*}
f_0(P) &=& f_0(K) + 1\\
f_k(P) &=& f_k(K) + f_{k-1}(K)\hspace{10pt}\text{ for }1\le k\le
d-2\\
f_{d-1}(P) &=& 1 + f_{d-2}(K) \, .
\end{eqnarray*}
\end{theorem}
We notice that the rows of Pascal's 3-triangle act in the same manner and we
claim the face numbers for $\P_{n}(\I_{1,n-1})$ can be derived from Pascal's
3-triangle.

\begin{proof}[Proof of Theorem \ref{fvector}]
Recall that $\mathcal{I} = \{e_1,e_2,\ldots,e_n,\alpha_{1,n-1},\alpha_{2,n}\}$ is the
vertex set for $\P_{n}(\I_{1,n-1})$ with $n\ge 3$, and let $\mathcal{R}_{k} :=
\operatorname{conv}(\mathcal{I}\setminus\{e_{k},e_{k+1},\ldots,e_{n-1}\})$ for
$1\le k<n$.  Then it is clear that $\P_{n}(\I_{1,n-1})$ is the convex hull of the
union of the $(n-1)$-dimensional polytope $\mathcal{R}_{n-1}$ and
$e_{n-1}\notin\operatorname{aff}(\mathcal{R}_{n-1})$ (by Corollary
\ref{dimsum}), and thus is a pyramid and its face numbers can be computed as in
Theorem \ref{facenumbers} from the face numbers of $\mathcal{R}_{n-1}$.

Notice next that $\mathcal{R}_{n-1}$ is the convex hull of the
$(n-2)$-dimensional polytope $\mathcal{R}_{n-2}$ and
$e_{n-2}\notin\operatorname{aff}(\mathcal{R}_{n-2})$ (again by Corollary
\ref{dimsum}), so we can compute the face numbers of $\mathcal{R}_{n-1}$ from
those of $\mathcal{R}_{n-2}$ as in Theorem \ref{facenumbers}.

We can continue this process until we get that $\mathcal{R}_{3}$ is the convex
hull of $\mathcal{R}_2$ and $e_2\notin\operatorname{aff}(\mathcal{R}_2)$.  But
we notice that $\mathcal{R}_2 = \mathcal{B}$, so by Lemma \ref{basefvector},
$f_0(\mathcal{R}_2) = f_1(\mathcal{R}_2) = 4$.  From here we can build the 
$f$-vector of $\P_{n}(\I_{1,n-1})$ recursively, using Theorem \ref{facenumbers}.
\end{proof}


Our next goal is to compute the volume of $\P_{n}(\I_{1,n-1})$.
A simple induction proof gives:

\begin{lemma}\label{A2}
The determinant of the $n\times n$-matrix
\[\begin{bmatrix}0 & 1 & 1 & \cdots & 1\\1 & 0 & 1 & \cdots & 1\\ & & \ddots & &
\\1 & \cdots & 1 & 0 & 1\\1 & 1 & \cdots & 1 & 0\end{bmatrix}\] is
$(-1)^{n-1}(n-1)$.\end{lemma}

\begin{proof}[Proof of Theorem \ref{volpyramid}]
In order to calculate the volume of $\P_{n}(\I_{1,n-1})$ we will first
triangulate the $2$-dimensional base of the pyramid $\mathcal{B}$ from Lemma
\ref{subdimension}: namely, $\mathcal{B}$ is the union of  
\[\triangle_1=\operatorname{conv}(
e_{1},e_{n},\alpha_{1,n-1}) \qquad \text{ and } \qquad \triangle_2=\operatorname{conv}(
e_{n},\alpha_{1,n-1},\alpha_{2,n}).\] 

\commentout{
Let $x$ be a point in the base, then for
some $\lambda_{i}\ge 0,$ where $\displaystyle{\sum_{i=1}^{4}\lambda_{i}=1},$
\begin{align*}
x&=\lambda_{1}e_{1}+\lambda_{2}e_{n}+\lambda_{3}\alpha_{1,n-1}+\lambda_{4}
\alpha_{2,n}\\
&=(\lambda_{1}+\lambda_{3},\lambda_{3}+\lambda_{4},\cdots,\lambda_{3}+\lambda_{4
},\lambda_{2}+\lambda_{4})\\
&=(\lambda_{1}-\lambda_{4})e_{1}+(\lambda_{2}+\lambda_{4})e_{n}+(\lambda_{3}
+\lambda_{4})\alpha_{1,n-1}\\
&=(\lambda_{1}+\lambda_{2})e_{n}+(\lambda_{1}+\lambda_{3})\alpha_{1,n-1}
+(\lambda_{4}-\lambda_{1})\alpha_{2,n}.
\end{align*}
So $x$ is a point in $\triangle_1$ if $\lambda_{1}\ge\lambda_{4}$ and $x$ is a
point in $\triangle_2$ if $\lambda_{4}\ge \lambda_{1}$. Thus $\triangle_1$ and
$\triangle_2$ is a triangulation of the $2$-dimensional base of the pyramid.
}

By Corollary \ref{dimsum}, each $e_{2},\dots,e_{n-1}$ adds a dimension so that
the convex hull of these points and $\triangle_1$ is an $n$-dimensional simplex.
The same can be said of $\triangle_2$. Call these simplices $S_1$ and $S_2$
respectively; thus $S_1$ and $S_2$ triangulate $\P_{n}(\I_{1,n-1})$, and
the sum of their volumes is equal to the volume of $\P_{n}(\I_{1,n-1})$. In order
to calculate the volume of $S_1$ and $S_2$, we will use the Cayley Menger
determinant \cite{MR2141304}. Consider $S_1$, whose volume is the
determinant of the matrix
$$\begin{bmatrix}e_1-\alpha_{1,n-1} & e_2-\alpha_{1,n-1} & \cdots &
e_{n}-\alpha_{1,n-1}\end{bmatrix}=\begin{bmatrix}0 & -1 & -1 & \cdots & -1 &
-1\\-1 & 0 & -1 & \cdots & -1 & -1\\-1 & -1 & 0 & -1 & \cdots & -1\\ & &  &
\ddots & & \\-1 & -1 & \cdots & -1 & 0 & -1\\0 & 0 & 0 & \cdots & 0 &
1\end{bmatrix}.$$
Cofactor expansion on the last row will leave us with the determinant, up to a
sign, of the $(n-1)\times(n-1)$ matrix
\begin{equation}\label{A}\begin{bmatrix}0 & -1 & -1 & \cdots & -1\\-1 & 0 & -1 &
\cdots & -1\\ & & \ddots & & \\-1 & \cdots & -1 & 0 & -1\\-1 & -1 & \cdots & -1
& 0\end{bmatrix},\end{equation}
which, when ignoring sign, by Lemma \ref{A2} is $n-2$. Therefore the volume of
$S_1$ is $n-2$.

\commentout{
Now consider the Cayley Menger determinant of $S_2$, the determinant of
$$\begin{bmatrix}\alpha_{1,n-1}-\alpha_{2,n} & e_{2}-\alpha_{2,n} &
e_3-\alpha_{2,n} & \cdots & e_n-\alpha_{2,n}\end{bmatrix}=\begin{bmatrix}1 & 0 &
0 & 0 & \cdots & 0 \\0 & 0 & -1 & -1 & \cdots & -1\\0 & -1 & 0 & -1 & \dots &
-1\\ & &  & \ddots & & \\0 & -1 & -1 & \cdots & 0 & -1 \\-1 & -1 & -1 & \cdots &
-1 & 0\end{bmatrix}.$$
By cofactor expansion on the first row we are left with the positive determinant
of the matrix (\ref{A}) which is $n-2$. Therefore 
}

A similar computation gives the volume of $S_2$ as $n-2$,
and so the volume of $\P_{n}(\I_{1,n-1})$ is $2(n-2),$ as desired.
\end{proof}

\section{The $i^\text{th}$ pyramidal interval vector polytope}\label{conjectures}

Recall that the $i^\text{th}$ pyramidal interval vector polytope is $\P_n(\I_{1,n-i})$, the convex hull of all interval vectors in $\RR^n$ with interval length $1$ or $n-i$. 

\begin{example}For $n=6$ and $i=2$,
\begin{align*}
  \P_6(\I_{1,4})
  &= \operatorname{conv}\big((1,0,0,0,0,0)\,,\,(0,1,0,0,
0,0,0)\,,\,(0,0,1,0,0,0)\,,\,(0,0,0,1,0,0)\,,\,(0,0,0,0,1,0)\,, \\
  &\qquad (0,0,0,0,0,1)\,,\,(1,1,1,1,0,0)\,,\,(0,1,1,1,1,0)\,,\,(0, 0,1,1,1,1)\big) \, .
\end{align*}
\end{example}

The following proposition collects certain properties of $\P_n(\I_{1,n-i})$. 
We omit its proof, since it is similar to the proofs in Section \ref{pyramid}.

\begin{proposition}
The dimension of $\P_n(\I_{1,n-i})$ is $n$.
Furthermore, $\P_n(\I_{1,n-i})$ can be constructed by taking iterative pyramids (with the
sequence of top vertices $e_{i+1},e_{i+2},\ldots,e_{n-i}$) over the
$2i$-dimensional base
\[\operatorname{conv}\left(\{e_1,e_2,\ldots,e_i,e_{n-i+1},\ldots,e_n
,\alpha_{1,n-i},\alpha_{2,n-i-1}\ldots,\alpha_{i+1,n}\}\right).\]
\end{proposition}

Using {\tt polymake} to generate $f$-vectors for varying $n$, we observed 
that the $f$-vectors of $\P_n(\I_{1,n-i})$
correspond to the sum of multiple shifted Pascal triangles; this is again due to its pyramid
property. 
We also offer the following:

\begin{conj}
The volume of $\P_n(\I_{1,n-i})$ equals $2^{i}(n-(i+1))$.
\end{conj}

We conjecture something more concrete: 
namely, that $\P_n(\I_{1,n-i})$
can be triangulated into $2^{i}$ simplices, and pyramiding over each of these
simplices each yields a volume of $n-(i+1)$.


\bibliographystyle{amsplain}  
\bibliography{MyBibliography}  

\end{document}